\documentclass[11pt]{article}
\usepackage{multirow}
\usepackage{authblk}
\usepackage{amsmath,amssymb,amsbsy,amsfonts,amsthm,latexsym,
            amsopn,amstext,amsxtra,euscript,amscd,color}
            
\PassOptionsToPackage{hyphens}{url}\usepackage{hyperref}
    
\usepackage[capbesideposition=outside,capbesidesep=quad]{floatrow}

 \DeclareMathOperator{\Av}{Av}
 
 \usepackage{fullpage, array}

\newtheorem{theorem}{Theorem}[section]
\newtheorem{corollary}{Corollary}[section]
\newtheorem{lemma}{Lemma}[section]

\theoremstyle{definition}

\newtheorem{remark}{Remark}[section]

\def\+{\oplus}

\renewcommand{\S}{\mathcal S}

\def\00{{\bf 0}}
\def\11{{\bf 1}}
\def\+{\oplus}

\def\\{\cr}
\def\({\left(}
\def\){\right)}

\newcommand{\C}{\mathcal{C}}
\renewcommand{\c}{c}

\usepackage{booktabs}

\usepackage[colorinlistoftodos,prependcaption,textsize=tiny]{todonotes}

\providecommand{\newoperator}[3]{%
  \newcommand*{#1}{\mathop{#2}#3}}

\newoperator{\FD}{\mathrm{FD}}{\nolimits}

\title{Powers of permutations that avoid chains of patterns}
\author[1]{Kassie Archer}
\author[1]{Aaron Geary}
  
\affil[1]{\small{Department of Mathematics, United States Naval Academy, Annapolis, MD, 21402, 

Email: karcher@usna.edu, geary@usna.edu}}

\date{}
\begin{document}

\maketitle
\begin{abstract}
In a recent paper, B\'{o}na and Smith define the notion of \textit{strong avoidance}, in which a permutation and its square both avoid a given pattern. In this paper, we generalize this idea to what we call \textit{chain avoidance}. We say that a permutation avoids a chain of patterns $(\tau_1 : \tau_2: \cdots : \tau_k)$ if the $i$-th power of the permutation avoids the pattern $\tau_i$. 
We enumerate the set of permutations $\pi$ which avoid the chain $(213, 312 : \tau)$, i.e.,~unimodal permutations whose square avoids $\tau$, for $\tau \in \S_3$ and use this to find a lower bound on the number of permutations that avoid the chain $(312: \tau)$ for $\tau \in \S_3$. We finish the paper by discussing permutations that avoid longer chains.

\end{abstract}

\section{Introduction and Background} 

Let $S_n$ denote the symmetric group on $[n]=\{1,\ldots,n\}$. Each element of this group is a bijection that can be written in its one-line notation as $\pi = \pi_1\pi_2\ldots\pi_n$ where $\pi_i=\pi(i)$ for all $i\in[n]$. From this perspective we can consider \emph{patterns} in the permutation. If $\pi =\pi_1\cdots\pi_n\in\S_n$ and $\tau=\tau_1\cdots\tau_m\in\S_m$, then we say $\pi$ contains the pattern $\tau$ if there exist entries $\pi_{i_1},\ldots,\pi_{i_m}$ such that $i_1<\cdots< i_m$ and  $\pi_{i_j}<\pi_{i_k}$ if and only if $\tau_j<\tau_k$ for all $j,k\in [m]$. If $\pi$ does not contain $\tau$ then we say $\pi$ \emph{avoids} $\tau$.

The notion of \emph{strong pattern avoidance} was recently defined in \cite{BS19}. A permutation is said to strongly avoid a given pattern $\tau$ if both $\pi$ and $\pi^2$ avoid $\tau$. In \cite{BS19}, the authors find that the enumeration of the set of permutations that strongly avoid 312 is given by the generating function: 
\[
S(x) = \dfrac{1-x-x^2+x^3}{1-2x-x^2+2x^3-x^4}
\]
and show that the number that strongly avoid $12\ldots k$ is eventually zero. They also find a lower bound for the number of permutations that strongly avoid 321, showing it has a growth rate of at least 2.3247. 

In \cite{BD20} the authors continue the study of strong avoidance. In particular, they enumerate the set of permutations that strongly avoid the sets $\{132, 3421\}$, $\{213, 4312\}$, and $\{321, 3412\}$. They also extend the definition of strong avoidance to what they call \textit{powerful avoidance} of patterns, in which a permutation and all its powers avoid a given pattern.
In \cite{P23}, the authors prove a conjecture from \cite{BD20}, namely that the number of permutations that strongly avoid the set $\{321, 1342\}$ is given by $2F_{n+1}-n-2$, where $F_{n+1}$ is the $(n+1)$-st Fibonacci number.

In this paper, we consider a generalization of this property, which we call \emph{chain pattern avoidance}, in which we say that a permutation $\pi$ avoids the chain $(\tau_1 : \tau_2: \cdots : \tau_k)$ if the $i$-th power of $\pi$ (i.e.,~$\pi^i$) avoids $\tau_i$.  We denote the set of permutations in $\S_n$ that avoid a given chain $(\tau_1 : \tau_2: \cdots : \tau_k)$ by $\C_n(\tau_1 : \tau_2: \cdots : \tau_k)$ and we denote the number of such permutations by $\c_n(\tau_1 : \tau_2: \cdots : \tau_k)$. 

Note that this definition can be extended to include sets of patterns as well, which we will separate with a comma. For example, if $\pi$ avoids the chain $(\sigma_1, \sigma_2 : \tau : \rho_1, \rho_2)$, then $\pi$ avoids both $\sigma_1$ and $\sigma_2$, $\pi^2$ avoids $\tau$, and $\pi^3$ avoids both $\rho_1$ and $\rho_2$. As another example, $\pi$ avoids the chain $(\sigma : \varnothing : \varnothing : \tau)$ if $\pi$ avoids $\sigma$ and $\pi^4$ avoids $\tau$. In this case, we place no requirement on pattern avoidance of $\pi^2$ or $\pi^3$ as indicated by using the empty set in the associated positions in the chain.

We can also extend this notion to chains of infinite length. For instance, we would say that a permutation $\pi$ avoids the chain $(\tau^{\infty})$ if $\pi^i$ avoids $\tau$ for all $i\geq 1$. This is equivalent to the recently introduced notion of \emph{powerful pattern avoidance} found in \cite{BD20}. If $\pi$ avoids $\tau$ and all higher powers $\pi^i$ avoid $\eta$ for $i\geq 2$, we would say $\pi$ avoids the chain $(\tau : \eta^{\infty})$. 

Our primary results in this paper include the enumeration of permutations on $n$ elements that avoid certain chains. We begin by investigating avoidance chains that start with $\pi$ avoiding the pair 231 and 321 and with $\pi^2$ avoiding a single permutation $\tau\in\S_3.$  A summary of these results can be found in Table~\ref{tab:unimodal}. We then use these results to find lower bounds on permutations that avoid 312 and whose squares avoids $\tau\in\S_3$. We find that for all such $\tau$ except 123, the number of permutations avoiding the chain $(312: \tau)$ is bounded below by $2^{n-1}$, and in some cases provide a better bound (see Table~\ref{table:312}). We then turn our attention to studying permutations that avoid longer chains. In particular, we enumerate permutations that avoid 231 and 321 and whose $k$-th power avoids a single pattern of length 3. Examples of these results appear in Table~\ref{table:231321}.

\section{Preliminaries}

The \emph{direct sum} of two permutations $\sigma\in\S_k$ and $\tau\in\S_\ell$ is the permutation $\pi= \sigma\oplus \tau\in\S_{k+\ell}$ defined by  \[\pi_i = \begin{cases}
    \sigma_i & \text{for } 1\leq i\leq k\\
    \tau_{i-k}+k & \text{for } k+1\leq i\leq k+\ell.
\end{cases}\]
For example, $52314\oplus 2341 = 523147896$. As another example, consider the composition $d= (2,1,5,3)$ of 11. Then, if $\delta_m$ denotes the decreasing permutation of length $m$, we have \[\bigoplus_{i=1}^4\delta_{d_i} = 21\oplus 1\oplus 54321\oplus 321 = 2\ 1\ 3\ 8\ 7\ 6\ 5\ 4\ 11\ 10\ 9.\]

The following lemma states that direct sums behave well under taking powers of permutations. 

\begin{lemma}\label{lem:directsumsquare}
Suppose $\sigma_i\in\S_{m_i}$ with $m_i\geq 1$ for each $1\leq i\leq k$ and let $m$ be any positive integer. If $\pi = \bigoplus_{i=1}^k \sigma_i$, then $\pi^m= \bigoplus_{i=1}^k \sigma_i^m$. 
\end{lemma} 

\begin{proof}
    If $k=1$ then $\pi=\sigma_1$, so assume $k>1$. Notice that for each $1\leq i \leq k$, if $\sum_{\ell=1}^{i-1} m_\ell< j \leq\sum_{\ell=1}^{i} m_\ell$,
    then $\sum_{\ell=1}^{i-1} m_\ell<\pi_j \leq\sum_{\ell=1}^{i} m_\ell$ also. In other words, all elements associated to $\sigma_i$ are mapped to each other under $\pi$, and thus under any powers of $\pi$. Therefore, the result must follow. 
\end{proof}

Similarly, the \emph{skew sum} of two permutations $\sigma\in\S_k$ and $\tau\in\S_\ell$ as the permutation $\pi= \sigma\ominus \tau\in\S_{k+\ell}$ defined by  \[\pi_i = \begin{cases}
    \sigma_i+\ell & \text{for } 1\leq i\leq k\\
    \tau_{i-k} & \text{for } k+1\leq i\leq k+\ell.
\end{cases}\]
For example, $52314\ominus 2341 = 967582341$.

\section{Avoiding the chain \texorpdfstring{$(312,213:\tau)$}{(312,213:tau)}}\label{312,213}

We start with the study of permutations in $\S_n$  that avoid both 213 and 312. These are exactly the set of unimodal permutations, i.e.,~those permutations $\pi = \pi_1\pi_2\ldots\pi_n$ so that there is some $k$ with $\pi_1<\pi_2<\cdots<\pi_k$ and $\pi_k>\pi_{k+1}>\cdots>\pi_n$. In this section, we enumerate permutations that avoid the chain $(213, 312 : \tau)$ for a given pattern $\tau\in \S_3$. A summary of these results can be found in Table~\ref{tab:unimodal}.

\renewcommand{\arraystretch}{2}
 \begin{table}[ht]
            \begin{center}
\begin{tabular}{| >{\centering\arraybackslash} m{1cm} | >{\centering\arraybackslash} m{3cm} | >{\centering\arraybackslash} m{3cm} | >{\centering\arraybackslash} m{3cm} |}
    \hline
    $\tau$ & $\c_n(213, 312 : \tau)$ & Theorem & OEIS \\
    \hline
    \hline
   123 &
   $\left\lfloor \dfrac{n}{2} \right\rfloor$ &
    Theorem~\ref{thm:uni-123}  &A004526 \\[4pt]
    \hline
    132 & $n+1$ & Theorem~\ref{thm:uni-132} &A000027 \\[4pt]
    \hline
    213 & $\left\lceil \dfrac{n(n+1)}{3} \right\rceil$ & Theorem~\ref{thm:uni-213}  & A007980\\[4pt]
    \hline
    231 & $\left\lceil \dfrac{n(n+1)}{3} \right\rceil$ & Theorem~\ref{thm:uni-231}  & A007980\\[4pt]
    \hline
    312 & $\left\lfloor \dfrac{n^2}{4} \right\rfloor + 1$ & Theorem~\ref{thm:uni-312} & A033638\\[4pt]
    \hline
    321 & $\left\lfloor \dfrac{(n+3)^2}{4} \right\rfloor - 5$ & Theorem~\ref{thm:uni-321}  & $\text{A002620}-5$ \\[4pt]
    \hline
\end{tabular}
            \end{center}
            \caption{Number of unimodal permutations whose square avoids a pattern of length 3 and the associated OEIS \cite{OEIS} number for each. }
            \label{tab:unimodal} 
      \end{table}

We begin with the case where $\tau=123$, i.e., those permutations that avoid the patterns 213 and 312 and whose square avoids the pattern 123. This is the most complicated proof in this section.
\begin{theorem}\label{thm:uni-123}
For  $n\geq 8$,  the number of permutations in $\S_n$ that avoid the chain $(213, 312:  123)$ is \[\c_n(213,312: 123) = \Big\lfloor \dfrac{n}{2}\Big\rfloor.\]
\end{theorem}

\begin{proof}
  Let $n\geq 8$.  
  Let us first show that $2\leq \pi_1\leq \lceil\frac{n+1}{2}\rceil$. First, we note that $\pi_1\neq 1$ since otherwise in order for $\pi^2$ to avoid 123, we would need $\pi^2 = 1n(n-1)\ldots 32$. This would imply that $\pi_2\neq 2$ but since $\pi$ is unimodal, would would have to have $\pi_n=2$. But this would imply that $2 = \pi^2_n = \pi_{\pi_n}=\pi_{2}$, which is a contradiction.
  We also cannot have $\pi_1> \lceil\frac{n+1}{2}\rceil$ since this would imply that $\pi_i=n$ for some $i<n+1-\pi_1$ since $\pi$ must end in $(\pi_1-1)(\pi_1-2)\ldots 21$. Therefore, in $\pi^2$, we would have $\pi^2_i=1$ and $\pi^2_n = \pi_1$. Since $\pi_1> \lceil\frac{n+1}{2}\rceil$ and $i< n+1- \lceil\frac{n+1}{2}\rceil$, there must be some element $\pi^2_j$ with $i<j<n$ so that $\pi^2_j<\pi_1$, implying that $1\pi^2_j\pi_1$ is a 123 pattern in $\pi^2$. 
   Thus $\pi_1 \in \{2,3,\ldots , \lceil\frac{n+1}{2}\rceil\}$. 
   

  Now, let us show that if $\pi_i=n$, then either either $i=\frac{n+1}{2}$ if $n$ is odd, or $i\in\{\frac{n}{2},\frac{n}{2}+1\}$ if $n$ is even. For sake of contradiction, suppose $n$ occurs before the $\lceil\frac{n}{2}\rceil$ position, say position $j$. Then $\pi^2$ has a 1 in $j$-th position. Since $n$ occurred before the $\lceil\frac{n}{2}\rceil$ position, then there is necessarily a decreasing sequence at least as long as $\lceil\frac{n}{2}\rceil$ at the end of $\pi$, and thus $\pi_{j+1}\geq \lfloor \frac{n}{2}\rfloor+1$ and $\pi_{j+2}\geq \lfloor \frac{n}{2}\rfloor$. But then $\pi^2_{j+1}<\pi^2_{j+2}$ and so $1\pi^2_{j+1}\pi_{j+2}^2$ is a 123 pattern. Thus $n$ cannot occur before the $\lceil\frac{n}{2}\rceil$-th position. 
  Now suppose $n$ occurs after the $\lceil\frac{n+1}{2}\rceil$-th position. Then we must have $\pi_1\leq \lfloor \frac{n-1}{2}\rfloor$ and $\pi_2 \leq \lfloor\frac{n+1}{2}\rfloor$. However, this implies that $\pi^2_1<\pi^2_2<n$, and so $\pi^2_1\pi_2^2n$ is a 123 pattern in $\pi^2.$ 
  
  We claim that if $n$ is odd, then $\pi$ must be of the form \[\pi = k\left(\frac{n+1}{2}+1\right)\left(\frac{n+1}{2}+2\right)\ldots n  \left(\frac{n+1}{2}\right)\left(\frac{n+1}{2}-1\right)\ldots (k+1)(k-1)\ldots  21\] with $k \in \{2,3,\ldots , \frac{n+1}{2}\}$. 
  Notice this is a unimodal permutation and in this form $n$ always occurs in the $(\frac{n+1}{2})$-th position. Then the square of $\pi$ would be \[\pi^2=\pi_k\left(\frac{n+1}{2}\right)\left(\frac{n+1}{2}-1\right)\left(\frac{n+1}{2}-2\right)\ldots 1n(n-1)\ldots\pi_{k+1}\pi_{k-1}\ldots\left(\frac{n+1}{2}+1\right)k   \] with $\pi_k \in \{\frac{n+1}{2}+1, \frac{n+1}{2}+2,\ldots,n \} $ and thus $\pi^2$ avoids 123. Note that there are exactly $\frac{n-1}{2}$ of these permutations.

  To prove this claim, we need only show that if $1<i<\frac{n+1}{2}$ and $j>\frac{n+1}{2}$, then $\pi_i>\pi_j.$ That is, every element before $n$ (except $\pi_1$) must be greater than every element after $n$. For the sake of contradiction, suppose this is not the case. Then, $j=\pi_2$ would be smaller than $m=\pi_{i+1}$, and so if $\pi_1=k$, we have $k<j<m$. Then, $\pi_k\pi_j\pi_m$ is a sequence in $\pi^2$. If this is not a 123 pattern, we need $j\geq \frac{n+1}{2}$. If $j>\frac{n+1}{2}$, $\pi$ must be of the form described above since $\pi_{\frac{n+1}{2}}=n,$ so $j=\frac{n+1}{2}$. In this case, either $\pi_m=m$ if $m=\frac{n+1}{2}+1$, or $\pi_m<\frac{n+1}{2}=j.$ But in both cases, $\pi^2_{\frac{n+1}{2}+2} = \pi_{\frac{n+1}{2}-1}>\pi_m$, and so $1\pi_m\pi_{\frac{n+1}{2}-1}$ is a 123 pattern.

   Similarly, we claim that if $n$ is even, then either \[\pi = k\left(\frac{n}{2}+1\right)\left(\frac{n}{2}+2\right)\ldots  n \left(\frac{n}{2}\right)\left(\frac{n}{2}-1\right)\ldots (k+1)(k-1)\ldots 1\] with $k \in \{2,3,\ldots , \frac{n}{2} \} $ and with \[\pi^2 = \pi_k  n\left( \frac{n}{2}\right)\left(\frac{n}{2}-1\right)\ldots  1 (n-1)(n-2)\ldots  \pi_{k+1}\pi_k\ldots\left(\frac{n}{2}+1\right) k\] which clearly avoids 123 since $\pi_k>\frac{n}{2}$ and $k\leq\frac{n}{2}$, or we have \[\pi = \left(\frac{n}{2}+1\right)\left(\frac{n}{2}+2\right) \ldots  n \left(\frac{n}{2}\right)\left(\frac{n}{2}-2\right)\ldots 321\] with \[\pi^2 = 
   \left(\frac{n}{2}-1\right)\left(\frac{n}{2}-2\right)\ldots 321n(n-1)\ldots\left(\frac{n}{2}+1\right)\] which also clearly avoids 123. Note that there are exactly $\frac{n}{2}$ of these permutations. 

A very similar argument proves that every element to the left of $n$, except possibly $\pi_1=k$, must be greater than every element to the right of $n.$ We need only show that if $\pi_1\in\{2,3,\ldots,\frac{n}{2}\}$, we must not have \[\pi = k\left(\frac{n}{2}+2\right)\left(\frac{n}{2}+3\right)\ldots  n \left(\frac{n}{2}+1\right)\left(\frac{n}{2}\right)\ldots (k+1)(k-1)\ldots 1,\] which is the only other possibility given the restrictions of where $n$ can occur in $\pi$. However, in this case, we have 
\[\pi^2 = \pi_k\left(\frac{n}{2}\right)\ldots (k+1)(k-1)\ldots 21 \left(\frac{n}{2}+1\right)n\ldots \pi_{k+1}\pi_k \ldots \left(\frac{n}{2}+2\right) k,\]
which has the 123 pattern $1(\frac{n}{2}+1)n$. Therefore, the proof is complete.
\end{proof}

Next, we consider unimodal permutations whose square avoids 132.

\begin{theorem}\label{thm:uni-132}
    For  $n\geq 3$,  the number of permutations in $\S_n$ that avoid the chain $(213, 312:  132)$ is $\c_n(213,312: 132) = n+1.$
\end{theorem}

\begin{proof}
    Suppose $\pi\in\C_n(213,312:132)$ has $\pi_1=1$. Then we must also have $\pi^2_1=1$, and thus $\pi^2$ must be the identity permutation. This means that $\pi$ must be an involution. However, the only involutions that are unimodal and start with $\pi_1=1$ are of the form: $\pi=12\ldots kn(n-1)\ldots (k+1)$ for some $k\geq 1$. Together with the involution $\pi=n(n-1)\ldots 21$, this gives us $n$ possible permutations in $\C_n(213,312:132)$. We claim that the only other permutation in $\C_n(213,312:132)$ is the permutation $\pi = 23\ldots n1$ whose square is $\pi^2=34\ldots n12$, which avoids 132.

     For the sake of contradiction, suppose $\pi$ satisfies the properties that $\pi_1\not\in\{1,n\}$, and thus $\pi_2>\pi_1$ and $\pi_n=1.$ Furthermore, we can assume $\pi_{n-1}\neq n$ since the only such permutation would be $\pi = 23\ldots n1$, as considered in the previous paragraph. If $\pi_1>2$, then $\pi_{n-1}\pi_n=21$, and so $1\pi_2\pi_1$ is a 132 pattern in $\pi^2.$ If $\pi_1=2,$ then $\pi^2_n=2$, and so the only way for $\pi^2$ to avoid a 132 pattern is if $\pi^{2}_{n-1}=1,$ but this would imply $\pi_{n-1}=n$ and so $\pi=23\ldots n1$, as already considered above.
\end{proof}

In the next theorem, we consider those unimodal permutations whose square avoids 213. Notice that these are exactly the permutations that avoid 312 and strongly avoid 213.

\begin{theorem}\label{thm:uni-213}
    For $n\geq 1$, the number of permutations in $\S_n$ that avoid the chain $(213, 312 : 213)$ is \[\c_n(213,312: 213) =\bigg\lceil \dfrac{n(n+1)}{3}\bigg\rceil.\]
\end{theorem}
\begin{proof}
    We will show that the number of permutations avoiding $(213, 312 : 213)$ satisfies the recurrence $c_n(213,312:213)=c_{n-1}(213,312:213)+\big\lfloor\frac{2n+1}{3}\big\rfloor$. First note that there are clearly $c_{n-1}(213,312:213)$ such permutations with $\pi_1=1$ since any permutation with this property must be of the form $1\oplus \pi'$ for any $\pi'\in\C_{n-1}(213,312:213).$ It remains to show that there are exactly $\big\lfloor\frac{2n+1}{3}\big\rfloor$ permutations in $\C_n(213,312:213)$ with $\pi_1\neq 1$ and thus $\pi_n=1$. 

    We first claim that any permutation avoiding the chain $(213,312 : 213)$ must be of the form $(\iota_a\oplus \delta_b)\ominus \delta_c$ for some $a+b+c=n$ where $\iota_a$ is the increasing permutation of length $a$ and $\delta_b$ and $\delta_c$ are the decreasing permutations of length $b$ and $c$, respectively. If $\pi$ were not of this form, there would be some  $k$ so that $\pi_j=k+1$, $\pi_i=n$ and $\pi_\ell=k$ with $1<j<i<\ell$. 
    However, we would then have the subsequence $\pi^2_1\pi^2_j\pi^2_i\pi^2_\ell = \pi_{\pi_1}\pi_{k+1}1\pi_k$ in $\pi^2$. In order to avoid 213, we would need both $\pi_{\pi_1}$ and $\pi_{k+1}$ to be greater than $\pi_k$. But since $\pi_1<k+1$, and thus $\pi_1<k$, this would imply that $\pi_{\pi_{1}}\pi_k\pi_{k+1}$ is either a 213 or 312 pattern in the original permutation $\pi,$ which is a contradiction.

    We next note that the unimodal permutation with $\pi_1=n$ (i.e. with $a=c=0$), is a permutation avoiding the chain $(213,312: 213)$ since $\pi=n(n-1)\ldots 21$ and $\pi^2=12\ldots (n-1)n.$

    Now suppose instead that $a,b,c>0$. Notice that since $\pi$ has the form $(\iota_a\oplus \delta_b)\ominus \delta_c$, it is enough to know $a$ and $c$ since $b=n-a-c$. In this case, $\pi_{a+1} = n$ and $\pi_1=c+1$.  Let us now write \[\pi = (c+1)(c+2)\ldots(c+a)n(n-1)\ldots (n-a-c+1)c(c-1)\ldots 21.\]
    Now, since $\pi^2$ must avoid 213 and $\pi_n=1$, we must have that the elements in $\{\pi_{c+1}, \ldots, \pi_{c+a}
    \}$ are the numbers $\{n-a+1, \ldots, n\}$. In particular, we must have that $\pi_{c+1}>\pi_c$, and so $c+1\leq a+1$ since $a+1$ is the position of $n$ in $\pi$. 

    Let us first consider the case where $c=a$. Then 
    $\pi_{c+1}\ldots\pi_{c+a}=n(n-1)\ldots (n-a-1)$ exactly when $b\geq c.$ Equivalently, this works when $1\leq a\leq \frac{n}{3}$, so there are $\lfloor\frac{n}{3}\rfloor$ permutations of this form. 

    Finally, consider the case where $c<a$. Then $\pi_{c+1}<n$. Since \[\pi = (c+1)(c+2)\ldots(c+a)n(n-1)\ldots (n-a-c+1)c(c-1)\ldots 21\] and  the elements in $\{\pi_{c+1}, \ldots, \pi_{c+a}
    \}$ must be the numbers in the interval $\{n-a+1, \ldots, n\}$, we must have that $n-a+1 = \pi_{c+1}=n-a+1$, which implies that $b=c$. Since $a>c$, this will only work if $1\leq a<\frac{n}{3}$ and so there are $\lfloor\frac{n-1}{3}\rfloor$ permutations of this form.

    Now, since we have $\lfloor \frac{n}{3}\rfloor + \lfloor \frac{n-1}{3}\rfloor + 1 = \lfloor \frac{2n+1}{3}\rfloor$ permutations of size $n$ with the property that $\pi_n=1$, the total number of permutations that avoid the chain $(213, 312: 213)$ satisfy the recurrence $c_n(213,312:213)=c_{n-1}(213,312:213)+\big\lfloor\frac{2n+1}{3}\big\rfloor$. Together with the easily-checked initial conditions, the result follows. 
\end{proof}

In the next theorem, we consider those permutations avoiding the chain $(213, 312: 231).$

\begin{theorem}\label{thm:uni-231}
    For $n\geq 1$, the number of permutations in $\S_n$ that avoid the chain $(213, 312 : 231)$ is \[\c_n(213,312: 231) =\bigg\lceil \dfrac{n(n+1)}{3}\bigg\rceil.\]
\end{theorem}

\begin{proof}
        We will show that the number of permutations avoiding $(213, 312 : 231)$ satisfies the recurrence $c_n(213,312:231)=c_{n-1}(213,312:231)+\big\lfloor\frac{2n-1}{3}\big\rfloor$. Together with the easily-checked initial conditions, the closed form in the statement of the theorem will follow. 
        
        First note that there are clearly $c_{n-1}(213,312:231)$ permutations in $C_n(213,312:231)$ with $\pi_1=1$ since any such permutation is of the form $1\oplus\pi'$ for some $\pi'\in\C_{n-1}(213,312:231).$ It is enough for us to show that there are exactly $\big\lfloor\frac{2n+1}{3}\big\rfloor$ permutations in $C_n(213,312:231)$ with $\pi_n=1$. 

        We first claim that any permutation avoiding the chain $(213,312 : 231)$ must be of the form $(\iota_a\oplus \delta_b)\ominus \delta_c$ for some $a+b+c=n$ where $\iota_a$ is the increasing permutation of length $a$ and $\delta_b$ and $\delta_c$ are the decreasing permutations of length $b$ and $c$, respectively. If $\pi$ were not of this form, there would be some  $k$ so that $\pi_j=k+1$, $\pi_i=n$ and $\pi_\ell=k$ with $1<j<i<\ell$. But then in $\pi^2$, we would have the subsequence $\pi^2_1\pi^2_j\pi^2_i\pi^2_\ell = \pi_{\pi_1}\pi_{k+1}1\pi_k$. 
        Since we necessarily have $i>2$ and $\pi^2_i=1$, we must have that $\pi^2_1\pi^2_2\ldots\pi^2_{i-1}$ is decreasing. This implies that $\{\pi^2_2,\ldots, \pi^2_{i-1}\}\subseteq\{\pi_{i+1}, \ldots\pi_{n}\}$. This furthermore implies that $k+1>i$, so $k\geq i$ and so $\pi_k>\pi_{k+1}.$ Since $\pi$ avoids 213 and 312, we must also have $\pi_1<\pi_{k+1}$. Since $\pi_1<\pi_{k+1}<\pi_k$, in $\pi^2$, we thus have $\pi^2_{j}\pi^2_{\ell}\pi^2_n = \pi_{k+1}\pi_k\pi_1$ is a 231 pattern, which is a contradiction.
        
        Next, let us notice that if $\pi_i=n$, then $i\leq\pi_1$. For the sake of contradiction, suppose $i>\pi_1$. Then $\pi_{\pi_1} = 2\pi_1-1$ since $\pi$ is of the form described in the previous paragraph. But then in $\pi^2,$ we have $\pi^2_1 = 2\pi_1-1\neq n$ and $\pi^2_n=\pi_1.$ Thus $(2\pi_1-1)n\pi_1$ is a 231 pattern in $\pi^2.$
        
        Consider the case when $\pi_i=n$. Then if $i<\pi_1$, we must have $\pi = \iota_i\ominus \delta_{n-i}$. That is, $\pi = (n-i+1)(n-i+2)\ldots (n-1)n(n-i)\ldots 21.$ Furthermore, if $\pi_1=(n-i-1)> i$, then each of these permutations avoids the chain $(213,312:231)$. 
        Indeed, if this were not the case, then we would have $i>1$, $\pi_i=n$, and $\pi_{i+1}=n-1$.
        Since $\pi$ is unimodal and starts with $\pi_1$, it must end with $(\pi_1-1)\ldots 21$. Therefore in $\pi^2,$ we have the sequence $\pi_{i+1}\pi_i\pi_{i-1} = (n-1)n\pi_{i-1}$, which must be a 231 pattern. 
        It is straightforward to see that if $\pi_1=(n-i-1)> i$, then $\pi = (n-i+1)(n-i+2)\ldots (n-1)n(n-i)\ldots 21$ avoids the chain $(213, 312:231)$ since $\pi^2$ is of the form $\delta_i\oplus\iota_{n-2i}\oplus \delta_{i}$, which has no 231 pattern. Thus there is one permutation for each $i<\frac{n-1}{2}$ and so there are $\lfloor \frac{n-1}{2}\rfloor$ such permutations.

        Finally, if $\pi_{\pi_1}=n,$ then there is exactly one permutation of the correct form, namely $\pi = \pi_1(\pi_1+1)\ldots(2\pi_1-2)n(n-1)\ldots(2\pi_1-1)(\pi_1-1)\ldots21$, and this permutation avoids the chain $(213, 312: 231)$ exactly when $\frac{n+2}{3}\leq\pi_1\leq \frac{n+2}{2}.$
        Indeed, if we must have $\pi_1\leq \frac{n+2}{2}$ in order for $\pi$ to simultaneously be of the correct form  and for 
        $\pi_{\pi_1}=n$.  Since $\pi = \pi_1(\pi_1+1)\ldots(2\pi_1-2)n(n-1)\ldots(2\pi_1-1)(\pi_1-1)\ldots21$, then $\pi^2$ must be equal to $\pi^2 = n(n-1)\ldots \pi_{2\pi_1-2}12\ldots \pi_{2\pi_1-1}\pi_{\pi_1-1}\ldots (\pi_1+1)\pi_1$. If $\frac{n+2}{3}>\pi_1,$ 
        then $n-\pi_1>2\pi_1-2$, which implies that $\pi^2$ has the form:
        \[\pi^2=\fbox{$n(n-1)\ldots(n-\pi_1+2)$}\fbox{$12\ldots (\pi-1)$}\fbox{$(2\pi_1-1)\ldots(n-\pi_1+1)$}\fbox{$(2\pi_1-2)\ldots (\pi_1+1)\pi_1$}\]
        where the blocks are sequences of consecutive numbers. 
        Notice that in $\pi^2,$ we have \[\pi^2_{n-\pi_1}\pi^2_{n-\pi_1+1}\pi^2_{n-\pi_1+2} = \pi_{2\pi-2}\pi_{2\pi-1}\pi_{\pi-1} =(n-\pi_1)(n-\pi_1+1)(2\pi_1-2)\] which is a 231 pattern since $2\pi_1-2<n-\pi_1$.

        Since we have shown that there is one permutation for each possible value of $\pi_1$ with $\frac{n+2}{3}<\pi_1\leq n$, we have $n-\lfloor\frac{n+1}{3}\rfloor = \lfloor \frac{2n+1}{3}\rfloor$ permutations with $\pi_n=1.$
\end{proof}

In the next theorem, we consider permutations that avoid the chain $(213,312:312)$. We could equivalently say that these permutations avoid 213 and strongly avoid 312. 

\begin{theorem}\label{thm:uni-312}
    For $n\geq 1$, the number of permutations in $\S_n$ that avoid the chain $(213, 312 : 312)$ is \[\c_n(213,312: 312)=\Big\lfloor \dfrac{n^2}{4}\Big\rfloor +1.\]
\end{theorem}

\begin{proof}
    We will show that the number of permutations avoiding $(213, 312 : 312)$ satisfies the recurrence $c_n(213,312:312)=c_{n-1}(213,312:312)+\big\lfloor\frac{n}{2}\big\rfloor$. Together with initial condition $c_3(213,312:312)=3$,  it follows that the number of permutations avoiding the chain $(213,312:312)$ has the closed form $c_n(213,312:312)=\big\lfloor \frac{n^2}{4}\big\rfloor +1$. 

 First, note that the permutations $\pi\in\C_n(213,312:312)$ with $\pi_1=1$ are of the form $\pi=1\oplus \pi'$ with $\pi'\in\C_{n-1}(213,312:312).$ Thus there are exactly $c_{n-1}(213,312:312)$ permutations $\pi\in\C_n(213,312:312)$ with $\pi_1=1$.
 
 Now assume $\pi$ is unimodal with $\pi_1=m>1$ and $\pi_i=n$. Let us first show that all elements that appear before $n$ are greater than all elements that appear after $n$. It is enough to show that $m=\pi_1$ is greater than $k=\pi_{i+1}$ since $m$ is the smallest element to the left of $n$ and $k$ is the largest element to the right of $n$. Since $\pi=m\ldots n k \ldots 1$, then $\pi^2=\pi_m\ldots1\pi_k\ldots m$. If $\pi^2$ avoids 312, we must have $\pi_m<m$. However, since $\pi_m<m$ and $m$ is the smallest element to the left of $n$, we must have $m>i$. But since $\pi$ is unimodal and $k>m$, we have $\pi_k<\pi_m$ and so $\pi_m1\pi_k$ is a 312 pattern in $\pi^2$. Thus $\pi$ must be of the form $m(m+1)\ldots n(m-1)\ldots 1$ for some $m\geq 2$.  

Suppose $m=\pi_1\leq \lceil \frac{n}{2}\rceil$. Since $\pi$ ends in 1, $\pi^2$ ends in $m$.  Also, since $m\leq \lceil \frac{n}{2} \rceil$, we have $\pi_m = 2m-1>m$. Thus in $\pi^2,$ we have the 312 pattern $\pi^2_1\pi^2_i\pi^2_n=\pi_m1m$.
Thus if $\pi$ ends in 1 it must be of the form $m(m+1)\ldots n(m-1)\ldots 1$ with $m > \lceil \frac{n}{2} \rceil$. Since \[\pi^2 = (n-m+1)(n-m)\ldots 21 (n-m+2)\ldots (m-1)n(n-1)\ldots m,\] which avoids 312, these permutations do indeed avoid the chain $(213,312:312)$. There are exactly $\lfloor \frac{n}{2} \rfloor$ of these permutations. Thus the recurrence $c_n(213,312:312)=c_{n-1}(213,312:312)+\big\lfloor\frac{n}{2}\big\rfloor$ holds, and the result follows.

\end{proof}

Finally, we consider permutations that avoid the chain $(213,312:321)$. 

\begin{theorem}\label{thm:uni-321}
    For $n\geq 3$, the number of permutations in $\S_n$ that avoid the chain $(213, 312 : 321)$ is \[\c_n(213,312: 321)=\Big\lfloor \dfrac{(n+3)^2}{4}\Big\rfloor -5.\]
\end{theorem}
\begin{proof}
    We will show that the number of permutations avoiding $(213, 312 : 321)$ satisfies the recurrence $c_n(213,312:321)=c_{n-1}(213,312:321)+\big\lceil\frac{n}{2}\big\rceil+1$, which together with the initial condition $c_3(213,312:321)=4$ has closed form $\big\lfloor \frac{(n+3)^2}{4}\big\rfloor -5$. 

    First, note that the permutations $\pi\in\C_n(213,312:321)$ with $\pi_1=1$ are of the form $\pi=1\oplus \pi'$ with $\pi'\in\C_{n-1}(213,312:321).$ Thus there are exactly $c_{n-1}(213,312:321)$ permutations $\pi\in\C_n(213,312:321)$ with $\pi_1=1$.
    Let us also note that the two permutations  $\pi = 23\ldots (n-1)n1$ with $\pi^2 = 34\ldots (n-1)n12$ and the permutation $\pi=n(n-1)\ldots 21$ with $\pi^2=12\ldots (n-1)n$ avoid the chain $(213,312:321).$
   
    Next we show that all other permutations in $\C_n(213,312: 321)$ must be of the form \[\pi=kn(n-1)\ldots (k+1)(k-1) \ldots 21\] with $\lfloor \frac{n}{2} \rfloor<k<n$. Since $k> \lfloor \frac{n}{2} \rfloor$, we must have $\pi_k\leq k+1$. Thus this permutation does avoid the chain $(213,312:321)$ since \[\pi^2 =\pi_k 12\ldots (\pi_k-1)(\pi_k+1)\ldots (k-1)(k+1)\ldots (n-1)nk\] must avoid 321. Furthermore, there are $\lfloor\frac{n-1}{2}\rfloor$ such permutations.

    Now, let us show that there are no other unimodal permutations whose square avoids 321. Since we have already considered the case where $\pi_1=n,$ let us assume $\pi_1\neq n$ and thus $\pi_1<\pi_2$. If $\pi_2=n,$ then either $\pi$ is of the form described above, or it is of the form $\pi=kn(n-1)\ldots 1$ with $k\leq\lfloor\frac{n}{2}\rfloor$. Then since $\pi_k$ must be greater than $k+1$, and  $\pi_1^2=\pi_k$ and $\pi^2_n=k$,  we are guaranteed a 321 pattern in $\pi^2$. 

    Now assume $\pi$ begins with $jk$ such that $2< j<k<n$. That is, $\pi=jk\ldots n\ldots21$. Then $\pi^2$ must end in $kj$ and so there is a 321 pattern in $\pi^2,$ namely $nkj$. Now suppose $\pi_1=2$ and $\pi_2<n$. We have already seen that $\pi=23\ldots(n-1)n1\in\C_n(213,312:321)$, so suppose that $b=\pi_{n-1}\neq n$.
    That is, $\pi=2 3 \ldots (\pi_b-1)\ldots n\ldots b1$ and thus $\pi^2$ ends in $\pi_b2$. Since $\pi^2$ ends with 2, the only way for it to avoid 321 is for all elements except 1 and 2 to be in increasing order, and in particular, we need $\pi_b=n$. 
    However, if $\pi_b=n$, we would have $\pi = 23\ldots (b-1)\pi_{b-1}n(n-1)\ldots b1$. 
    This implies that $\{\pi^2_1, \ldots, \pi^2_{b-3}\} = \{3,4,\ldots, b-1\}$, $\pi^2_b=1$,  $\pi^2_n=2$, and either $\pi^2_{b-1}$ or $\pi^2_{b+1}$ is equal to $b$ since $\pi_{n-1}=b$.
    And so, we must have $\pi^2_{b-2}>b$ and thus $\pi^2_{b-2}b2$ is a 321 pattern in $\pi^2.$

    Therefore there are $\big\lfloor\frac{n-1}{2}\big\rfloor+2 = \big\lceil\frac{n}{2}\big\rceil+1$ permutations which do not start with 1 that avoid the chain $(213, 312 : 321)$, and so the claim is proven.
\end{proof}

\section{Avoiding the chain \texorpdfstring{$(312:\tau)$}{(312:tau)}}\label{312}

In this section, we find a lower bound for the number of permutations that avoid the chain $(312:\tau)$ for $\tau\in\S_3$.

First let us observe that if $\pi\in\Av_n(312)$, then $\pi = \alpha\oplus\beta$ where $\alpha\in\Av_i(312)$ for $1\leq i\leq n$ with $\alpha_i=1$ and $\beta\in\Av_{n-i}(312)$. Therefore, in this case $\pi^2 = \alpha^2\oplus\beta^2$. 
For convenience, we will define $a_n(\sigma:\tau)$ to be the number of permutations $\pi\in\S_n$ avoiding the chain $(\sigma:\tau)$ that have the additional property that $\pi_n=1$ (i.e., the permutation ends in 1).

In \cite{BS19}, B\'{o}na and Smith find the exact number of permutations that avoid the chain $(312:312)$ (equivalently, that strongly avoid the pattern 312). They do this by recognizing the permutation as a direct sum $\pi = \alpha\oplus\beta$ where $\alpha$ ends in 1, as described in the paragraph above. Since any 312 pattern must appear completely in the $\alpha^2$ or the $\beta^2$ part of $\pi^2$, knowing the number of permutations strongly avoiding 312 that end in 1 is enough to enumerate all permutations that strongly avoid 312. They find that the number of strongly 312-avoiding permutations that end in 1 are exactly the unimodal permutations whose square avoids 312 that end in 1. Using these facts, they prove the following theorem. 

\begin{theorem}[\cite{BS19}]\label{thm:bona-smith}
    For $n\geq 1$, the generating function for $c_n(312:312)$ is given by \[b(x) = \dfrac{1 - x - x^2 + x^3}{1 - 2 x - x^2 +2 x^3 - x^4}.\]
\end{theorem}

In this section, we find that the unimodal permutations whose square avoids the pattern $\tau$ and that end in 1 do not completely describe all permutations ending in 1 and avoiding the chain $(312:\tau)$ for $\tau\in\S_3\setminus \{312\}$, as they did in the 312 case. However, these permutations do allow us to find a lower bound for the enumeration of $\C_n(312:\tau)$. 
In Table~\ref{table:312}, we give the initial values of $c_n(312:\tau)$ for $\tau\in\{132,213,231,321\}$ along with the initial values of the lower bound found in this section.

\renewcommand{\arraystretch}{2}
 \begin{table}[ht]
            \begin{center}
\begin{tabular}{|c|cccccccccc|c|c|}
    \hline
      $n$ & 3& 4& 5& 6& 7 &8 &9&10&11&12 & Theorem\\ \hline\hline
    $c_n(312:132)$ & 5& 11& 23& 49& 102& 206& 419& 849& 1704& 3420 
& \multirow{ 2}{*}{Theorem~\ref{thm:312-132}}  \\ \cline{1-11}
    $2^{n-1}$ & 4 & 8 & 16 & 32 & 64 & 128 & 256 & 512 & 1024 & 2048& \\ \hline\hline
	$c_n(312:213)$ & 5& 11& 23& 49& 102& 206& 419& 849& 1704& 3420 & \multirow{ 2}{*}{Theorem~\ref{thm:312-213}}\\ \cline{1-11}
 $\lfloor \frac{10}{7}\cdot 2^{n-1}\rfloor$ & 5& 11& 22& 45& 91& 182& 365& 731& 1462& 2925 &\\ \hline\hline
	$c_n(312:231)$ & 5& 13& 30& 70& 167& 395& 932 &2206 &5217& 12334 & \multirow{ 2}{*}{Theorem~\ref{thm:312-231}}\\ \cline{1-11}
	 $b_n$ & 4& 9& 19& 41& 87& 186& 396 & 845 & 1801 & 3841 &\\ \hline\hline
	$c_n(312:321)$ &5& 12& 29& 68& 160& 378& 891& 2101& 4954& 11683 & \multirow{ 2}{*}{Theorem~\ref{thm:312-321}}\\ \cline{1-11}
 $d_n$ &5& 11& 25& 56& 126& 283& 636& 1429& 3211& 7215 &\\ \hline
\end{tabular}
            \end{center}
            \caption{Values of $c_n(312:\tau)$ for $\tau\in\{132,213,231,321\}$ and $3\leq n\leq 12$, together with their lower bounds as given in the corresponding theorems.}
            \label{table:312} 
      \end{table}

Let us first consider the case where $\tau=123$. In this case, a decomposition of $\pi$ into a direct sum is not possible since all such permutations must end in 1. Thus the lower bound attained in this case is obtained directly from Theorem~\ref{thm:uni-123}.

\begin{theorem}\label{thm:312-123}
Let $n\geq 8.$ Then $c_n(312:123)\geq \lfloor \frac{n}{2}\rfloor.$
\end{theorem}

This case certainly appears to be the strangest, as it appears the sequence is not even increasing. The sequence $c_n(312:123)$ for $n\in[12]$ is given by: 
\[
1,2,1,4,7,12,11,29,26,50,41,108
\]
and it is not immediately clear how to attain a recursive lower bound, as we have for other $\tau\in\S_3$ in the remainder of this section.

Next, let us consider the case when $\tau=132$. In this case, we can decompose the permutation as described above and can use Theorem~\ref{thm:312-132} to attain a lower bound for $c_n(312:132).$ Recall that $a_n(\sigma:\tau)$ is the number of permutations in $\C_n(\sigma:\tau)$ that end with $\pi_n=1$.

\begin{lemma}\label{lem:312-132}
Let $n\geq 3$. Then,
\[
c_n(312:132) = a_n(312:132)+\sum_{i=1}^{n-1} a_i(312:132)\cdot2^{n-1-i} .
\]
\end{lemma}

\begin{proof}
    Suppose $\pi= \alpha\oplus\beta\in\Av_n(312)$. 
Let us first notice that the only way for a 132 pattern to appear in $\pi^2 = \alpha^2\oplus\beta^2$ is for there to either be a 132 pattern in $\alpha^2$ or in $\beta^2$, or for the ``1'' of the 132 pattern to be in $\alpha^2$ and the ``3'' and ``2'' to be in $\beta^2$. 
Therefore, in order for $\pi^2$ to avoid 132, we need $\alpha^2$ to avoid 132 and for $\beta^2$ to be the increasing permutation (since $\alpha^2$ is non-empty). Thus $\alpha$ is a permutation avoiding the chain $(312:132)$ that ends in 1 and $\beta$ is an involution avoiding $312$. 
Since any involution avoiding 312 of length $n$ must be of the form $\bigoplus_{i=1}^k \delta_{m_i}$ for some composition $m_1+m_2+\ldots+m_k=n$, where $\delta_{m}=m(m-1)\ldots 321$, these permutations are in bijection with compositions, and thus there are $2^{n-1}$ such permutations. 

Since there are $a_n(312:132)$ elements in $\C_n(312:132)$ that have $1$ in the $n$-th position and $a_i(312:132)\cdot 2^{n-i-1}$ such permutations that have 1 in the $i$-th position, the result follows. 
\end{proof}

\begin{theorem}\label{thm:312-132}
Suppose $n\geq 3$. Then,
$c_n(312:132)\geq 2^{n-1}$.
\end{theorem}
\begin{proof}
    We know from the proof of Theorem~\ref{thm:uni-132} that there is always one permutation avoiding the chain $(312,213:132)$ and ending in 1. Since this permutation must also avoid the chain $(312:132)$, we know there is always at least one permutation that avoids the chain $(312:132)$ and that ends in 1. Thus, by Lemma~\ref{lem:312-132}, $c_n(312:132)\geq 1+\sum_{i=1}^{n-1} 2^{n-1-i} = 2^{n-1}.$
\end{proof}

Now, let us consider the case when $\tau=213.$

\begin{lemma}\label{lem:312-213}
Let $n\geq 3$. Then,
\[
c_n(312:213) =  a_n(312:213)+\sum_{i=1}^{n-1} a_i(312:213)\cdot 2^{n-1-i}.
\]
\end{lemma}
\begin{proof}
    Suppose $\pi= \alpha\oplus\beta\in\Av_n(312)$. Let us first notice that the only way for a 213 pattern to appear in $\pi^2=\alpha^2\oplus\beta^2$ is for there to either be a 213 pattern in $\alpha^2$ or in $\beta^2$, or for the ``2'' and ``1'' of the 213 pattern to be in $\alpha^2$ and the ``3'' to be in $\beta^2$. Therefore, in order for $\pi^2$ to avoid 213, we need $\beta^2$ to avoid 213 and for $\alpha^2$ to avoid 213, and furthermore  we need $\alpha^2$ to be the increasing permutation when $\beta$ is nonempty. Thus, if $\beta$ is nonempty, then $\alpha$ is an involution avoiding $312$ the ends in $1$. However, the only involution that avoids 312 and ends in 1 is the decreasing permutation, so $\alpha$ must be decreasing when $\beta$ is nonempty. Therefore, either $\beta$ is empty, in which case the 1 appears at the end of the permutation and so there are $a_n(312:213)$ choices for $\alpha$, or $\beta$ is nonempty, in which case the 1 appears in the $i$-th position for some $1\leq i\leq n-1$, in which case there is one choice for $\alpha$ and $c_{n-i}$ choices for $\beta$. By reindexing, we get: \[c_n(312:213) =  a_n(312:213)+\sum_{i=1}^{n-1} c_i(312:213).\] By applying this recurrence to $c_{n-1}(312:213)$, we get \[c_n(312:213) =  a_n(312:213) + a_{n-1}(312:213)+\sum_{i=1}^{n-2} 2c_i(312:213).\] Applying it again to $c_{n-2}$ we get: \[c_n(312:213) =  a_n(312:213) + a_{n-1}(312:213)+2a_{n-2}(312:213) + \sum_{i=1}^{n-3} 4c_i(312:213).\] Continuing in this way to eliminate all copies of $c_i(312:213)$ from the right hand side, we get the recurrence in the lemma. 
\end{proof}

\begin{theorem}\label{thm:312-213}
Let $n\geq 5$. Then,
$c_n(312:213)\geq \lfloor\frac{10}{7}\cdot 2^{n-1}\rfloor$.
\end{theorem}
\begin{proof}
    We know from the proof of Theorem~\ref{thm:uni-213} that there are $\lfloor\frac{2n+1}{3}\rfloor = \lceil\frac{2n-1}{3}\rceil$ permutations avoiding the chain $(312,213:213)$ and ending in 1. Since these permutations must also avoid the chain $(312:213)$, we know $a_n(312:213)\geq \lceil\frac{2n-1}{3}\rceil$ for all $n$. Thus, by Lemma~\ref{lem:312-213},
    \begin{align*}c_n(312:213)  &= a_n(312:213)+\sum_{i=1}^{n-1} 2^{n-1-i}a_i(312:213) \\
    & \geq \bigg\lceil\frac{2n-1}{3}\bigg\rceil+\sum_{i=1}^{n-1} 2^{n-1-i}\cdot\bigg\lceil\frac{2i-1}{3}\bigg\rceil \\ 
    &=\bigg\lceil\frac{2n-1}{3}\bigg\rceil+ \sum_{\substack{i\in[n-1]\\ i\neq2 \bmod 3}} \sum_{j=0}^{n-i-1} 2^{j}\\
    &=\bigg\lceil\frac{2n-1}{3}\bigg\rceil+ \sum_{i=1}^{n-1} (2^{n-i}-1)- \sum_{j=1}^{\lfloor\frac{n}{3}\rfloor} (2^{n-3j+1}-1)\\
    &=\bigg\lceil\frac{2n-1}{3}\bigg\rceil+ (2^n-n-1)- \bigg(\bigg\lfloor\frac{2^{n+1}}{7}\bigg\rfloor - \bigg\lfloor\frac{n+1}{3}\bigg\rfloor\bigg)\\
    &= \bigg\lfloor\frac{10}{7}\cdot 2^{n-1}\bigg\rfloor,\end{align*}
    and so the theorem is proven.
\end{proof}

\begin{remark}\label{remark:rci}
    In \cite{BS19}, the authors prove that a permutation $\pi$ strongly avoids $\sigma$ if and only if $\pi^{-1}$ avoids $\sigma^{-1}$ and if and only if the reverse-complement $rc(\pi)$ strongly avoids $rc(\sigma)$. Their argument extends to chain avoidance, and so since $312=rc(312)^{-1}$ and $213=rc(132)^{-1}$, it must follow that $c_n(312:213) = c_n(312:132)$. Therefore, we can improve Theorem~\ref{thm:312-132} using Theorem~\ref{thm:312-213} and thus state that $c_n(312:132)\geq \lfloor\frac{10}{7}\cdot 2^{n-1}\rfloor$ as well.
\end{remark}

Next, we consider the case where $\tau=231.$

\begin{lemma}\label{lem:312-231}
Suppose $n\geq 3$. Then,
\[
c_n(312:231) =\sum_{i=1}^{n} a_i(312:231)c_{n-i}(312:231).
\]
\end{lemma}

\begin{proof}
    Suppose we have $\pi= \alpha\oplus\beta\in\Av_n(312)$. Let us first notice that the only way for a 231 pattern to appear in $\pi^2=\alpha^2\oplus\beta^2$ is for there to either be a 231 pattern completely in $\alpha^2$ or $\beta^2$. Therefore, for $\pi^2$ to avoid 213, we need $\beta^2$ to avoid 231 and for $\alpha^2$ to avoid 213 and end in 1.
The statement in the lemma follows. 
\end{proof}

\begin{theorem}\label{thm:312-231}
Suppose $n\geq 1$ and that $b_n$ is the $n$-th coefficient of the generating function \[b(x) = \dfrac{1 - x - x^2 + x^3}{1 - 2 x - x^2 +2 x^3 - x^4}.\] Then $c_n(312:231)\geq c_n$. 
\end{theorem}
\begin{proof}
    We know from the proof of Theorem~\ref{thm:uni-231} that there are $\lfloor\frac{n}{2}\rfloor$ permutations avoiding the chain $(312,213:321)$ with $\pi_n=1$ when $n\geq 2$ and 1 when $n=1$.
    Thus, by Lemma~\ref{lem:312-231},
    \begin{align*}c_n(312:231)  &= \sum_{i=1}^{n} a_i(312:231)c_{n-i}(312:231)\\ & \geq c_{n-1}(312:231) + \sum_{i=2}^{n} \bigg\lfloor\frac{i}{2}\bigg\rfloor \cdot c_{n-i}(312:231).\end{align*}
   
    Let  $b_n$ be as the sequence satisfying the recurrence $b_n = b_{n-1} + \sum_{i=2}^{n} \lfloor\frac{i}{2}\rfloor \cdot b_{n-i}$ with the initial condition that $b_0=b_1=1$. It is clear that in this case, $c_n(312:231)\geq b_n.$ It remains to show that the sequence $b_n$ has the generating function as stated in the theorem.
    
    Since we can write \[f(x)=
    x+\sum_{n\geq 2} \bigg\lfloor \frac{n}{2} \bigg\rfloor x^n = \dfrac{x-x^3+x^4}{(1-x)(1-x^2)},
    \] the generating function $b(x)$ for the sequence $b_n$ is  given by $b(x) = \dfrac{1}{1-f(x)}$, and thus $b(x)$ is exactly the generating function given in the theorem statement.
\end{proof}

\begin{corollary}
    For $n\geq 1$, $c_n(312:231)\geq c_n(312:312).$
\end{corollary}
\begin{proof}
    This follows immediately from Theorem~\ref{thm:312-231} since $b(x)$ is exactly the generating function for the number of permutations that strongly avoid 312, or equivalently avoid the chain $(312:312)$, as given in Theorem~\ref{thm:bona-smith} (\cite{BS19}).
\end{proof}

\begin{remark}The growth rate of $b_n$ in the statement of Theorem~\ref{thm:312-231} is approximately 2.13224, which implies the growth rate of $c_n(312:231)$ is at least this much.
\end{remark}

Finally, we consider the case where $\tau=321.$

\begin{lemma}\label{lem:312-321}
Suppose $n\geq 3$. Then,
\[
c_n(312:321) =\sum_{i=1}^{n} a_i(312:321)c_{n-i}(312:321).
\]
\end{lemma}
\begin{proof}
    Suppose we have $\pi= \alpha\oplus\beta\in\Av_n(312)$. Let us first notice that the only way for a 321 pattern to appear in $\pi^2=\alpha^2\oplus\beta^2$ is for there to be a 321 pattern either completely in $\alpha^2$ or $\beta^2$. Therefore, for $\pi^2$ to avoid 321, we need $\beta^2$ to avoid 321 and for $\alpha^2$ to avoid 321 and end in 1.
Therefore, the statement in the lemma follows. 
\end{proof}

\begin{theorem}\label{thm:312-321}
Suppose $n\geq 1$ and  $d_n$ is the $n$-th coefficient of the generating function \[d(x) = \dfrac{1-x-x^2+x^3}{1-2x-x^2+x^3}.\] Then $c_n(312:321)\geq d_n.$
\end{theorem}
\begin{proof}
    We know from the proof of Theorem~\ref{thm:uni-321} that there are $\lceil\frac{n}{2}\rceil+1$ permutations that avoid the chain $(312,213:321)$ and end in 1 when $n\geq 4$. Furthermore, $a_n(312:321)\geq \lceil\frac{n}{2}\rceil$ for all $n\geq1.$
    Thus, by Lemma~\ref{lem:312-321},
    \begin{align*}c_n(312:321)  & =\sum_{i=1}^{n} a_{i}(312:321)c_{n-i}(312:321) \\
    &\geq\sum_{i=1}^{n} \bigg\lceil\frac{i}{2}\bigg\rceil \cdot c_{n-i}(312:321).\end{align*}
    Through a calculation similar to the one in the proof of Theorem~\ref{thm:312-231}, we can find that the $d_n$ defined in the statement of the theorem is exactly the sequence satisfying $d_n = \sum_{i=1}^{n} \lceil\frac{i}{2}\rceil \cdot d_{n-i}$ with the initial conditions that $d_i=c_i(312:321)$ for $i\in\{0,1\}$. Therefore, we have that $a_n\geq d_n.$
\end{proof}

\begin{remark}
The growth rate of $c_n$ in the statement of Theorem~\ref{thm:312-321} is approximately 2.24598, which implies the growth rate of $c_n(312:321)$ is at least this much.
\end{remark}
\section{Avoiding longer chains of patterns}\label{231,321}

In this section we consider permutations $\pi$ that avoid 231 and 321 and investigate avoidance properties of their higher powers. In particular, we enumerate permutations avoiding the chain $(231, 321:\varnothing^{k-2}:\tau)$ for $k\geq 2$, where $\varnothing^{k-2}$ indicates a sequence of $k-2$ copies of $\varnothing$. In other words, these are the permutations $\pi$ avoiding 231 and 321 so that $\pi^k$ avoids a pattern $\tau$.  A summary of some of these results can be found in Table \ref{table:231321}.

 \begin{table}[htbp]
            \begin{center}
               \begin{tabular}{|c|c|c||c|c|}
               \hline
               $\tau_1$ & $\tau_2$ & $\tau_3$ & $\c_n(231,321:\tau_1:\tau_2:\tau_3)$ & OEIS \\
               \hline\hline
                132 & $\varnothing$& $\varnothing$ & $F_{n+2}-1 = \displaystyle\sum_{m=0}^{n-1}\sum_{i=0}^{\lfloor\frac{m}{2}\rfloor} \binom{m-i}{i}$  & A000071\\
               \hline  
                $\varnothing$&132 &  $\varnothing$ &$\displaystyle\sum_{m=0}^{n-1}\sum_{i=0}^{\lfloor\frac{m}{3}\rfloor} \binom{m-2i}{i}$ & A077868 \\
               \hline  
               $\varnothing$& $\varnothing$ &  132 &$\displaystyle\sum_{m=0}^{n-1}\sum_{i=0}^{\lfloor\frac{m}{4}\rfloor}\sum_{j=0}^{\lfloor\frac{m-4i}{2}\rfloor} \binom{m-3i-j}{i}\binom{m-4i-j}{j}$ & A368299 \\
               \hline  \hline
                231 & $\varnothing$& $\varnothing$ &  $F_{n+1}$ & A000045 \\
               \hline  
                $\varnothing$&231 &  $\varnothing$ & $T_{n+2}$ & A000073\\
               \hline  
               $\varnothing$& $\varnothing$ &  231 &$Q_{n+3}$ & A000078  \\
               \hline  \hline
               312 & $\varnothing$& $\varnothing$ &  $T_{n+2}$ & A000073\\
               \hline  
                $\varnothing$&312 &  $\varnothing$ & $Q_{n+3}$ &  A000078\\
               \hline  
               $\varnothing$& $\varnothing$ &  312 & {\footnotesize $\displaystyle\sum_{i=0}^{\lfloor\frac{n}{5}\rfloor}\sum_{j=0}^{\lfloor\frac{m-5i}{4}\rfloor}\sum_{r=0}^{\lfloor\frac{m-5i-4j}{2}\rfloor} \binom{n-4i-3j-r}{i}\binom{m-5i-3j-r}{j}\binom{m-5i-4j-r}{r}$}& A079976\\
               \hline  \hline
               132 & 231 & $\varnothing$ & $2F_n$ & A006355\\
               \hline
                 132 & $\varnothing$ & 231 & $2F_{n+1} - F_n$ & A000032\\ \hline
                 $\varnothing$ & 132 & 231 & $\displaystyle\sum_{k=0}^{\lceil\frac{2n}{3}\rceil} \binom{n-k}{\lfloor\frac{k}{2}\rfloor}$ & A097333 \\ \hline
            \end{tabular}
            \end{center}
            \caption{Number of permutations on $n$ elements satisfying various chains of length 2, 3, or 4 of the form $(231,321:\tau_1:\tau_2:\tau_3)$. Here, $F_n$ is the $n$-th Fibonacci number, $T_n$ is the $n$-th Tribonacci number, and $Q_n$ is the $n$-th Tetranacci number.}
            \label{table:231321}
      \end{table}

We start this section by noting that for a permutation $\pi$ to avoid 231 and 321 it must be of the form $\pi = \bigoplus_{i=1}^m \epsilon_{d_i}$ where $\epsilon_{d} :=d123\ldots (d-1)$. Then by Lemma \ref{lem:directsumsquare}, the $k$-th power of $\pi$ is
\begin{equation}\label{eq:1}
    \pi^k = \bigoplus_{i=1}^m \epsilon_{d_i}^k
\end{equation}
with the $k$-th power of $\epsilon_d$ equal to  $\epsilon_{d}^k = \iota_a\ominus \iota_b$ where $\iota_a$ and $\iota_b$ are the increasing permutations of length $a$ and $b$, respectively, and where $a=k \bmod d$ and $b=d-a.$



First, we find that the number of permutations that avoid the chain $(231, 321 : \varnothing^{k-2}: 123)$ is eventually zero.

\begin{theorem}
    For  $n\geq 5$,  the number of permutations in $\S_n$ that avoid the chain $(231,321:  123)$ is $\c_n(231,321:  123) = 0$. 
\end{theorem}

\begin{proof}
   Since $\pi = \bigoplus_{i=1}^m \epsilon_{d_i}$ and $\pi^k = \bigoplus_{i=1}^m \epsilon_{d_i}^k$, there is necessarily an ascent between each $\epsilon_{d_i}^k$. Thus we must have $m\leq 2$. If $m=1$, then $\pi=\epsilon_n$. Notice for $n \geq 5$ that $\epsilon_n^k$ is a skew sum of two increasing permutations and one of them must be at least length 3, so $\pi^k$ contains a 123 pattern. If $m=2$, then $\pi=\epsilon_{d_1}\oplus\epsilon_{d_2}$ where $d_1+d_2=n$. But since there must be an ascent between $\epsilon_{d_1}^k$ and $\epsilon_{d_2},$ each $\epsilon_{d_i}^k$ must be decreasing if $\pi$ avoids 123. But $\epsilon_d^k$ can only be the decreasing permutation if $d=1$ or $d=2$. Thus the largest permutation avoiding 231 and 321 whose $k$-th power avoids 123 is size 4. 
\end{proof}

\begin{theorem} \label{th:132-k}
 Denote by $c_n^{(k)}=c_n(231, 321 :\varnothing^{k-2}: 132)$, i.e., the number of permutations $\pi$ so that $\pi$ avoids 231 and 321 and $\pi^k$ avoids 132. Then for $n\geq 2$ and $k\geq 2$, $c_n^{(k)}$ is given by the recurrence: 
\[\c_n^{(k)}= 1+\sum_{d\mid k} \c_{n-d}^{(k)},\]
with the initial conditions that $c_1^{(k)}=1$ and $c_{i}^{(k)}=0$ for $i\leq 0$. The generating function $c^{(k)}_{132}(x) = \sum_{n\geq 1} c_n^{(k)} x^n$ is given by 
\[
c^{(k)}_{132}(x) = \frac{x}{(1-x)(1-\sum_{d\mid k}x^d)}.
\]
\end{theorem}

\begin{proof}
Assume $\pi = \bigoplus_{i=1}^m \epsilon_{d_i}$ avoids 231 and 321. Then $\pi^k= \bigoplus_{i=1}^m \epsilon_{d_i}^k$ as in (\ref{eq:1}). Since $\epsilon_{d}^k$ for must be of the form $\iota_a\ominus\iota_b$ for some $a+b=d$, the only way a 132 pattern in $\pi^k$ can appear is if the ``1'' is in some $\epsilon_{d_i}$ and the ``3'' and ``2'' is in another $\epsilon_{d_j}$ for some $i<j$.

Therefore, in order for $\pi^k$ to avoid 132, we must have that $\epsilon_{d_i}^k$ is increasing for all $2\leq i\leq m$. Thus $(d_1, d_2, \ldots, d_m)$ is a composition of $n$ with the property that $d_i\mid k$ for all $i\geq 2$ since $\epsilon_d^k$ is the identity permutation exactly when $d\mid k$. It is clear that compositions of this form satisfy this recurrence in the statement of the theorem since to obtain all such compositions of $n$, we can either add a new element equal to a divisor of $k$ to a previous composition or take the composition of $n$ of length 1. 
\end{proof}

To see an example of the recurrence for compositions of $n$ into divisors of $k$ (with the exception of the first element of the composition that can be any number), consider the case where $k=15$ and $n=68$. In this case, we could take any valid composition $(d_1, \ldots, d_m)$ of 53 and add ``15'' to get a composition $(d_1, \ldots, d_m, 15)$  of 68. Or we could take any valid composition $(d_1, \ldots, d_m)$ of 63 and add ``5'' to get a composition $(d_1, \ldots, d_m, 5)$ of 68. Likewise we could take any valid composition $(d_1, \ldots, d_m)$ of 65 and add ``3'' to get a composition $(d_1, \ldots, d_m, 3)$ of 68. We could also take any valid composition $(d_1, \ldots, d_m)$ of 67 and add ``1'' to get a composition $(d_1, \ldots, d_m, 1)$ of 68. Finally, we could take the length 1 composition $(68)$ of 68. Therefore the recurrence would be $c_{68} = 1+c_{67}+c_{65}+c_{63}+c_{53}$.

\begin{corollary} \label{th:132}
For $n\geq 3$, the number of permutations in $\S_n$ which avoid the chain $(231, 321 : 132)$ is $\c_n(231,321:  132) = F_{n+2}-1$
where $F_n$ is the $n$-th Fibonacci number.
\end{corollary}

\begin{proof}
As stated in the proof of Theorem~\ref{th:132-k}, in order for $\pi^2$ to avoid 132, we must have $(d_1, \ldots, d_k)$ satisfy that $d_i=1$ or $2$ for all $i\geq 2$. The number of such compositions of $n$ can be shown to be equal to $F_{n+2}-1$. Indeed, it is well-known that $F_{n+2}$ counts the compositions of $n+1$ into parts of size 1 and 2, and so $F_{n+2}-1$ is the number of compositions of $n+1$ into parts of size 1 and 2 that contains at least one occurrence of 2. Let $(c_1,c_2,\ldots,c_{n+2})$ be such a composition with $c_t$ equal to the first  occurrence of a 2. Then by taking $d_1 = (\sum_{i=1}^{t-1} c_i)+1$, and $d_j=c_{t+j-1}$ for all $2\leq j\leq n+1-t$, the composition $(d_1,d_2,\ldots, d_n)$ is a compositions with the property that $d_i=1$ or $2$ for all $i\geq 2$. Since this can be reversed by splitting $d_1$ into $(d_1-1)$ 1's followed by a 2, this is a bijection. \end{proof}

In Table~\ref{table:231321}, we include the case where $k=2, 3,$ or 4. In each case, we find a closed form by enumerating the number of compositions described in the proof of Theorem~\ref{th:132-k}. 

\begin{remark}
Using the ideas described in Remark~\ref{remark:rci}, we can note that since $231=(rc(231))^{-1}$, $321=(rc(321))^{-1}$, and $213 =(rc(132))^{-1}$, we have that $c_n(231, 321: \varnothing^{k-2}:213) = c_n(231,321:\varnothing^{k-2}:132)$ for all $k\geq 2$ and $n\geq 1$. 
\end{remark}

Next, let us consider the case where $\tau=231$.

\begin{theorem} \label{th:231-k}
 Denote by $c_n^{(k)}=c_n(231, 321 :\varnothing^{k-2}: 231)$, i.e., the number of permutations $\pi$ so that $\pi$ avoids 231 and 321 and $\pi^k$ avoids 231. Also let $D(k)$ be the set $D(k)= \{d: d\mid k\} \cup \{d' : d'\mid k-1\},$ i.e., the set of divisors of $k$ and $k-1$. Then for $n\geq 2$ and $k\geq 2$, $c_n^{(k)}$ is given by the recurrence: 
\[\c_n^{(k)}= \sum_{d\in D(k)} \c_{n-d}^{(k)},\]
with the initial conditions that $c_0^{(k)}=a_1=1$ and $c_{i}^{(k)}=0$ for $i<0$. The generating function $c^{(k)}_{231}(x) = \sum_{n\geq 0} c_n^{(k)} x^n$ is given by 
\[
c^{(k)}_{231}(x) = \frac{1}{(1-\sum_{d\in D(k)}x^d)}.
\]
\end{theorem}

\begin{proof}
Assume $\pi = \bigoplus_{i=1}^m \epsilon_{d_i}$ avoids 231 and 321. Then $\pi^k= \bigoplus_{i=1}^m \epsilon_{d_i}^k$ as in (\ref{eq:1}). Since $\epsilon_{d}^k$ for must be of the form $\iota_a\ominus\iota_b$ for some $a+b=d$, the only way a 231 pattern in $\pi^k$ can appear is the entire 231 pattern is in a single $\epsilon_{d_i}^k$ for some $i$. Notice that $\epsilon_{d}^k=\iota_a\ominus\iota_b$ avoids 231 exactly when $a=0$ or $a=1$. This happens exactly when $k=0\bmod d$ or $k=1\bmod d$. Therefore, $\pi=\bigoplus_{i=1}^m \epsilon_{d_i}$ avoids the chain $(231,321:231)$ exactly when each $d_i$ is either a divisor of $k$ or a divisor of $k-1$. 
Thus $\C_n(231,321:231)$ is in bijection with compositions of $n$ that only use divisors of $k$ or divisors of $k-1$, and the result follows.
\end{proof}

The following corollary deals with the case where $k=2$. The permutations in this case are exactly those permutations that avoid 321 and strongly avoid 231. 
\begin{corollary}
For $n\geq 3$, the number of permutations on $[n]$ which avoid the chain $(231, 321 : 231)$ is $\c_n(231,321:  231) = F_{n+1}$  where $F_n$ is the $n$-th Fibonacci number.
\end{corollary}

\begin{proof}
In this case, $\C_n(231,321:231)$ are in bijection with compositions of $n$ into 1s and 2s. Together with the initial conditions, this gives us the $(n+1)$-st Fibonacci number. 
\end{proof}

The case for $\tau=312$ is very similar to the previous case. In this case, we need that the powers of any blocks $\epsilon_d$ avoid 312. 

\begin{theorem} \label{th:312-k}
 Denote by $c_n^{(k)}=c_n(231, 321 :\varnothing^{k-2}: 312)$, i.e., the number of permutations $\pi$ so that $\pi$ avoids 231 and 321 and $\pi^k$ avoids 312. Also let $D'(k)$ be the set $D'(k)= \{d: d\mid k\} \cup \{d' : d'\mid k+1\},$ i.e., the set of divisors of $k$ and $k+1$. Then for $n\geq 2$ and $k\geq 2$, $c_n^{(k)}$ is given by the recurrence: 
\[\c_n^{(k)}= \sum_{d\in D'(k)} \c_{n-d}^{(k)},\]
with the initial conditions that $a_0=a_1=1$ and $a_{i}=0$ for $i<0$. The generating function $c^{(k)}_{312}(x) = \sum_{n\geq 1} c_n^{(k)} x^n$ is given by 
\[
c^{(k)}_{312}(x) = \frac{1}{(1-\sum_{d\in D'(k)}x^d)}.
\]
\end{theorem}

\begin{proof}
Assume $\pi = \bigoplus_{i=1}^m \epsilon_{d_i}$ avoids 231 and 321. Then $\pi^k= \bigoplus_{i=1}^m \epsilon_{d_i}^k$ as in (\ref{eq:1}). As before, since $\epsilon_{d}^k$ for must be of the form $\iota_a\ominus\iota_b$ for some $a+b=d$, the only way a 312 pattern in $\pi^k$ can appear is the entire 312 pattern is in a single $\epsilon_{d_i}^k$ for some $i$. Notice that $\epsilon_{d}^k=\iota_a\ominus\iota_b$ avoids 312 exactly when $b=0$ or $b=1$. This happens exactly when $k=0\bmod d$ or $k=d-1\bmod d$. Therefore, $\pi=\bigoplus_{i=1}^m \epsilon_{d_i}$ avoids the chain $(231,321:312)$ exactly when each $d_i$ is either a divisor of $k$ or a divisor of $k+1$. 
Thus $\C_n(231,321:312)$ is in bijection with compositions of $n$ that only use divisors of $k$ or divisors of $k+1$, and the result follows.
\end{proof}

A straightforward corollary of this is that the number of permutations avoiding 213 and 321 whose square avoids 312 is given by the ($n+2$)-nd Tribonacci number, $T_{n+2}$, defined by the recurrence $T_0=T_1=0$, $T_2=1$, and $T_n=T_{n-1}+T_{n-2}+T_{n-3}.$
\begin{corollary}
For $n\geq 3$, the number of permutations on $[n]$ which avoid the chain $(231, 321 : 312)$ is $\c_n(231,321:  312) = T_{n+2}$
where $T_n$ is the $n$-th Tribonacci number.
\end{corollary}


For the case that $\tau=321$, it is clear that all powers of $\pi\in\S_n$ that avoid 231 and 321 must avoid 321, as described in the proof below.
\begin{theorem}
For $n\geq 3$ and $k\geq 2$, the number of permutations on $[n]$ which avoid the chain $(231, 321 : \varnothing^{k-2}: 321)$ is $\c_n(231,321 : \varnothing^{k-2}: 321) = 2^{n-1}$.
\end{theorem}

\begin{proof}
The number of permutations in $\S_n$ elements avoiding both 231 and 321 is $2^{n-1}$ since they are clearly in bijection with compositions of $n$ (see also \cite{SS85}). Suppose $\pi = \bigoplus_{i=1}^k \epsilon_{d_k}$. We know that for $\epsilon_d^k = \iota_a\ominus\iota_b$ for some $a+b=d$, and in particular, $\epsilon_{d_i}^k$ avoids $321$ for each $i.$ Since a 321 pattern cannot occur within each block of the form $\epsilon_{d_i}^k$ and cannot occur between blocks in the direct sum, no power of $\pi$ can contain a 321. 
\end{proof}

This furthermore implies that all $\pi\in\S_n$ that avoid $231$ and $321$ also powerfully avoid 321, i.e., they avoid the chain  $(231, 321 : 321^{\infty})$.

The ideas in this section can be extended to include other chains as well. For example, the number of permutations $\pi\in\S_n$ that avoid 231 and 312, and where $\pi^j$ avoids 132 and $\pi^k$ avoids 231 is equal to the number of compositions $(d_1, d_2, \ldots, d_n)$ of $n$ where $d_1$ is a divisor of $j$ or $j-1$ and $d_i$ is both a divisor of $i$ and a divisor of $j$ or $j-1$. This can lead to some interesting solutions, a few of which are demonstrated in Table~\ref{table:231321}.

\section{Concluding Remarks and Open Questions}

A natural next direction is to consider longer patterns. There are many avoidance chains that seem to result in nice answers. For example, the authors conjecture that the number of permutations avoiding the chain $(231, 1432:231)$ (i.e., those that avoid 1432 and strongly avoid 231) is given by $L_{n+1}-\lceil\frac{n}{2}\rceil-1$ where $L_{n+1}$ is the ($n+1$)-st Lucas number (OEIS A000032). It would also be natural to consider consecutive patterns as part of the chain. For example, we conjecture that the number of unimodal permutations of length $n$ whose square avoids the consecutive pattern $\overline{213}$, that is, those that avoid the chain $(213, 312: \overline{213})$, is equal to $2^{n-2}+n-1.$

The notion of chain avoidance generalizes the idea of strong and powerful avoidance from \cite{BS19} and \cite{BD20}, respectively, but also generalizes another notion: pattern avoiding permutations of a given order in the group $\S_n$. 
There are many well known results (see for example, \cite{SS85, DRS07}) involving pattern-avoiding involutions, which in the language of chain avoidance could be described as avoiding the chain $(\tau:21)$ for some $\tau$ or as a pattern-avoiding permutation of order 2, but one could consider higher orders as well. It was posed as an open question in \cite{BS19} to enumerate permutations of order 3 avoiding 132, which can be described as avoiding the chain $(132: \varnothing: 21)$.  Generally, it is an open question to enumerate permutations avoiding a given pattern $\tau$ of a order $k\geq 3$ (corresponding to avoiding the chain $(\tau: \varnothing^{k-2}:21)$) with the exception of $\tau=231$ and $k=3$ \cite{BD20}.


\end{document}